\documentclass[reqno]{amsart}

\usepackage{amsmath,amssymb,amsthm,amsfonts,mathrsfs}
\usepackage{fullpage}
\usepackage{xcolor}
\usepackage[colorlinks=true,
linkcolor=blue,
anchorcolor=blue,
citecolor=red
]{hyperref}
\allowdisplaybreaks 
\newtheorem{theorem}{Theorem}[section]
\newtheorem{lemma}[theorem]{Lemma}

\newtheorem*{conjecture*}{Conjecture}
\theoremstyle{definition}

\theoremstyle{remark}

\newtheorem*{remark*}{remark}

\usepackage{colonequals}

\author{Runbo Li}
\address{The High School Affiliated to Renmin University of China International Curriculum Center, Beijing 100080, People's Republic of China}
\email{carey.lee.0433@gmail.com}

\makeatletter
\@namedef{subjclassname@2020}{\textup{}2020 Mathematics Subject Classification}
\makeatother

\title[]{A remark on large even integers of the form $p+P_3$}

\begin{document}
	
\begin{abstract}
Let $N$ denotes a sufficiently large even integer, $p$ denotes a prime and $P_{r}$ denotes an integer with at most $r$ prime factors. In this paper, we study the solutions of the equation $N-p=P_3$ and consider two special cases where $p$ is small, and $p,P_3$ are within short intervals.
\end{abstract}

\maketitle

\tableofcontents

\section{Introduction}
Let $N$ denotes a sufficiently large even integer, $p$ denotes a prime, and let $P_{r}$ denotes an integer with at most $r$ prime factors counted with multiplicity. For each $N \geqslant 4$ and $r \geqslant 2$, we define 
\begin{equation}
D_{1,r}(N):= \left|\left\{p : p \leqslant N, N-p=P_{r}\right\}\right|.
\end{equation}

In 1966 Jingrun Chen \cite{Chen1966} proved his remarkable Chen's theorem: let $N$ denotes a sufficiently large even integer, then
\begin{equation}
D_{1,2}(N) \geqslant 0.67\frac{C(N) N}{(\log N)^2}
\end{equation}
where
\begin{equation}
C(N):=\prod_{\substack{p \mid N \\ p>2}} \frac{p-1}{p-2} \prod_{p>2}\left(1-\frac{1}{(p-1)^{2}}\right).
\end{equation}
and the detail was published in \cite{Chen1973}. The original proof of Jingrun Chen was simplified by Pan, Ding and Wang \cite{PDW1975}, Halberstam and Richert \cite{HR74}, Halberstam \cite{Halberstam1975}, Ross \cite{Ross1975}. As Halberstam and Richert indicated in \cite{HR74}, it would be interesting to know whether a more elaborate weighting procedure could be adapted to the purpose of (2). This might lead to numerical improvements and could be important. Chen's constant 0.67 was improved successively to
$$
0.689, 0.7544, 0.81, 0.8285, 0.836, 0.867, 0.899
$$
by Halberstam and Richert \cite{HR74} \cite{Halberstam1975}, Chen \cite{Chen1978_1} \cite{Chen1978_2}, Cai and Lu \cite{CL2002}, Wu \cite{Wu2004}, Cai \cite{CAI867} and Wu \cite{Wu2008} respectively.

Chen's theorem with small primes was first studied by Cai \cite{Cai2002}. For $0<\theta \leqslant 1$, we define 
\begin{equation}
D_{1,r}^{\theta}(N):= \left|\left\{p : p \leqslant N^{\theta}, N-p=P_{r}\right\}\right|.
\end{equation}
Then it is proved in \cite{Cai2002} that for $0.95 \leqslant \theta \leqslant 1$, we have
\begin{equation}
D_{1,2}^{\theta}(N) \gg \frac{C(N) N^{\theta}}{(\log N)^2}.
\end{equation}
Cai's range $0.95 \leqslant \theta \leqslant 1$ was extended successively to $0.945 \leqslant \theta \leqslant 1$ in \cite{CL2011} and to $0.941 \leqslant \theta \leqslant 1$ in \cite{Cai2015}.

Chen's theorem in short intervals was first studied by Ross \cite{Ross1978}. For $0<\kappa \leqslant 1$, we define 
\begin{equation}
D_{1,r}(N,\kappa):= \left|\left\{p : N/2-N^{\kappa} \leqslant p,P_r \leqslant N/2+N^{\kappa}, N=p+P_{r}\right\}\right|.
\end{equation}
Then it is proved in \cite{Ross1978} that for $0.98 \leqslant \kappa \leqslant 1$, we have
\begin{equation}
D_{1,2}(N,\kappa) \gg \frac{C(N) N^{\kappa}}{(\log N)^2}.
\end{equation}
The constant 0.98 was improved successively to
$$
0.974, 0.973, 0.9729, 0.972, 0.971, 0.97
$$
by Wu \cite{Wu1993} \cite{Wu1994}, Salerno and Vitolo \cite{SV1993}, Cai and Lu \cite{CL1999}, Wu \cite{Wu2004} and Cai \cite{CAI867} respectively.

In this paper, we aim to relax the number of prime factors of $N-p$,  and at the same time extend the range of $\theta$. Our improvement partially relies on the cancellation of the use of Wu's mean value theorem. Our main result is the following theorem.
\begin{theorem}\label{t1}
for $0.838 \leqslant \theta \leqslant 1$ and $0.919 \leqslant \kappa \leqslant 1$, we have
$$
D_{1,3}^{\theta}(N) \gg \frac{C(N) N^{\theta}}{(\log N)^2} \quad \operatorname{and} \quad D_{1,3}(N,\kappa) \gg \frac{C(N) N^{\kappa}}{(\log N)^2}.
$$
\end{theorem}

We also generalize our results to integers of the form $ap+bP_3$. For two relatively prime square-free positive integers $a$ and $b$, let $M$ denotes a sufficiently large integer that is relatively prime to both $a$ and $b$, $a,b < M^{\varepsilon}$ and let $M$ be even if $a$ and $b$ are both odd. Let $R_{a,b}^{\theta}(M)$, $R_{a, b}(M,\kappa)$, $R_{a,b}^{\theta}(M,c,d)$ and $R_{a,b}(M,c,d,\kappa)$ denote the number of primes similar to those of \cite{LRB} but satisfy $\frac{M-ap}{b}=P_3$ instead of $P_2$. By using similar arguments as in \cite{LRB}, we prove that
\begin{theorem}\label{t2}
For $0.838 \leqslant \theta \leqslant 1$, $0.919 \leqslant \kappa \leqslant 1$ and $c \leqslant (\log N)^{C}$ where $C$ is a positive constant, we have
$$
R_{a, b}^{\theta}(M) \gg \frac{M^{\theta}}{a b(\log M)^{2}}, \quad R_{a, b}(M,\kappa) \gg \frac{M^{\kappa}}{a b(\log M)^{2}},
$$
$$
R_{a,b}^{\theta}(M,c,d) \gg \prod_{\substack{p \mid c \\ p \nmid M\\ p>2}} \left(\frac{p-1}{p-2}\right)\frac{M^{\theta}}{\varphi(c)ab(\log M)^2}
$$
and
$$
R_{a,b}(M,c,d,\kappa) \gg \prod_{\substack{p \mid c \\ p \nmid M\\ p>2}} \left(\frac{p-1}{p-2}\right)\frac{M^{\kappa}}{\varphi(c)ab(\log M)^2}.
$$
\end{theorem}
Since the detail of the proof of Theorem~\ref{t2} are similar to those of \cite{LRB} and  Theorem~\ref{t1} so we omit it in this paper. 

\section{Preliminary lemmas}
Let $\mathcal{A}$ denote a finite set of positive integers, $\mathcal{P}$ denote an infinite set of primes and $z \geqslant 2$. Suppose that $|\mathcal{A}| \sim X_{\mathcal{A}}$ and for square-free $d$, put

$$
\mathcal{P}=\{p : (p, N)=1\},\quad
\mathcal{P}(r)=\{p : p \in \mathcal{P},(p, r)=1\},
$$

$$
P(z)=\prod_{\substack{p\in \mathcal{P}\\p<z}} p,\quad
\mathcal{A}_{d}=\{a : a \in \mathcal{A}, a \equiv 0(\bmod d)\},\quad
S(\mathcal{A}; \mathcal{P},z)=\sum_{\substack{a \in \mathcal{A} \\ (a, P(z))=1}} 1.
$$

\begin{lemma}\label{l1} ([\cite{Kan2}, Lemma 1]). If
$$
\sum_{z_{1} \leqslant p<z_{2}} \frac{\omega(p)}{p}=\log \frac{\log z_{2}}{\log z_{1}}+O\left(\frac{1}{\log z_{1}}\right), \quad z_{2}>z_{1} \geqslant 2,
$$
where $\omega(d)$ is a multiplicative function, $0 \leqslant \omega(p)<p, X>1$ is independent of $d$. Then
$$
S(\mathcal{A}; \mathcal{P}, z) \geqslant X_{\mathcal{A}} W(z)\left\{f\left(\frac{\log D}{\log z}\right)+O\left(\frac{1}{\log ^{\frac{1}{3}} D}\right)\right\}-\sum_{\substack{n\leqslant D \\ n \mid P(z)}}\eta(X_{\mathcal{A}}, n)
$$
$$
S(\mathcal{A}; \mathcal{P}, z) \leqslant X_{\mathcal{A}} W(z)\left\{F\left(\frac{\log D}{\log z}\right)+O\left(\frac{1}{\log ^{\frac{1}{3}} D}\right)\right\}+\sum_{\substack{n\leqslant D \\ n \mid P(z)}}\eta(X_{\mathcal{A}},  n)
$$
where
$$
W(z)=\prod_{\substack{p<z \\ (p,N)=1}}\left(1-\frac{\omega(p)}{p}\right),\quad \eta(X_{\mathcal{A}}, n)=\left||\mathcal{A}_n|-\frac{\omega(n)}{n} X_\mathcal{A}\right|=\left|\sum_{\substack{a \in \mathcal{A} \\ a \equiv 0(\bmod n)}} 1-\frac{\omega(n)}{n} X_\mathcal{A}\right|,
$$
$\gamma$ denotes the Euler's constant, $f(s)$ and $F(s)$ are determined by the following differential-difference equation
\begin{align*}
\begin{cases}
F(s)=\frac{2 e^{\gamma}}{s}, \quad f(s)=0, \quad &0<s \leqslant 2,\\
(s F(s))^{\prime}=f(s-1), \quad(s f(s))^{\prime}=F(s-1), \quad &s \geqslant 2 .
\end{cases}
\end{align*}
\end{lemma}

\begin{lemma}\label{l2} ([\cite{CAI867}, Lemma 2], deduced from \cite{HR74}).
\begin{align*}
F(s)=&\frac{2 e^{\gamma}}{s}, \quad 0<s \leqslant 3;\\
F(s)=&\frac{2 e^{\gamma}}{s}\left(1+\int_{2}^{s-1} \frac{\log (t-1)}{t} d t\right), \quad 3 \leqslant s \leqslant 5 ;\\
F(s)=&\frac{2 e^{\gamma}}{s}\left(1+\int_{2}^{s-1} \frac{\log (t-1)}{t} d t+\int_{2}^{s-3} \frac{\log (t-1)}{t} d t \int_{t+2}^{s-1} \frac{1}{u} \log \frac{u-1}{t+1} d u\right), \quad 5 \leqslant s \leqslant 7;\\
f(s)=&\frac{2 e^{\gamma} \log (s-1)}{s}, \quad 2 \leqslant s \leqslant 4 ;\\
f(s)=&\frac{2 e^{\gamma}}{s}\left(\log (s-1)+\int_{3}^{s-1} \frac{d t}{t} \int_{2}^{t-1} \frac{\log (u-1)}{u} d u\right), \quad 4 \leqslant s \leqslant 6 ;\\
f(s)=& \frac{2 e^{\gamma}}{s}\left(\log (s-1)+\int_{3}^{s-1} \frac{d t}{t} \int_{2}^{t-1}\frac{\log (u-1)}{u} d u\right.\\ & \left.+\int_{2}^{s-4} \frac{\log (t-1)}{t} d t \int_{t+2}^{s-2} \frac{1}{u} \log \frac{u-1}{t+1} \log \frac{s}{u+2} d u\right), \quad 6 \leqslant s \leqslant 8.
\end{align*}
\end{lemma}

\begin{lemma}\label{mean} ([\cite{Perelli1984}, Theorem]).
For any given constant $A>0$, there exists a constant $B=B(A)>0$ such that
$$
\sum_{d \leqslant x^{t-1/2} (\log x)^{-B}} \max _{x/2 \leqslant y \leqslant x}\max _{(l, d)=1}\max _{h \leqslant x^t} \left|\pi(y+h;d,l)-\pi(y;d,l)-\frac{h}{\varphi(d)}\right| \ll \frac{x^t}{\log ^{A} x},
$$
where
$$
\frac{3}{5}<t \leqslant 1.
$$
\end{lemma}

\begin{lemma}\label{Wfunction}
If we define the function $\omega$ as $\omega(p)=0$ for primes $p \mid N$ and $\omega(p)=\frac{p}{p-1}$ for other primes and $N^{\frac{1}{\alpha}-\varepsilon}<z \leqslant N^{\frac{1}{\alpha}}$, then we have
$$
W(z)=\frac{2\alpha e^{-\gamma} C(N)(1+o(1))}{\log N}.
$$
\end{lemma}
\begin{proof} By similar arguments as in \cite{Cai2002}, we have
\begin{align}
\nonumber W(z)&=\prod_{p \mid N} \frac{p}{p-1} \prod_{(p,N)=1} \left(1-\frac{\omega(p)}{p}\right)\left(1-\frac{1}{p}\right)^{-1} \frac{\alpha e^{-\gamma}(1+o(1))}{\log N} \\
\nonumber & =\frac{2\alpha e^{-\gamma} C(N)(1+o(1))}{\log N}.
\end{align}
\end{proof}

\section{Proof of Theorem 1.1}
Let $\theta=0.838$ and $\kappa=0.919$ in this section. Put
$$
\mathcal{A}=\{N-p : p \leqslant N^\theta \} \quad \operatorname{and} \quad \mathcal{B}=\{N-p : N/2-N^\kappa \leqslant p \leqslant N/2+N^\kappa \}.
$$
Clearly we have
\begin{align}
D_{1,3}^{\theta}(N) &\geqslant S\left(\mathcal{A}; \mathcal{P}, N^{\frac{1}{11.99}}\right)-\frac{1}{2} \sum_{\substack{N^{\frac{1}{11.99}} \leqslant p<N^{\frac{1}{3}} \\ (p, N)=1}}S\left(\mathcal{A}_{p}; \mathcal{P}, N^{\frac{1}{11.99}}\right)=S_1-\frac{1}{2}S_2, \\
D_{1,3}(N,\kappa) &\geqslant S\left(\mathcal{B}; \mathcal{P}, N^{\frac{1}{11.99}}\right)-\frac{1}{2} \sum_{\substack{N^{\frac{1}{11.99}} \leqslant p<N^{\frac{1}{3}} \\ (p, N)=1}}S\left(\mathcal{B}_{p}; \mathcal{P}, N^{\frac{1}{11.99}}\right)=S^{\prime}_1-\frac{1}{2}S^{\prime}_2.
\end{align}

Now we define the function $\omega$ as $\omega(p)=0$ for primes $p \mid N$ and $\omega(p)=\frac{p}{p-1}$ for other primes. We can take 
$$
X_{\mathcal{A}}\sim \frac{N^{\theta}}{\theta \log N} \quad \operatorname{and} \quad X_{\mathcal{B}}\sim \frac{2N^{\kappa}}{\log N}.
$$

By Lemmas~\ref{l1}--\ref{Wfunction}, Bombieri's theorem and some routine arguments, we have
\begin{align}
S_1 &\geqslant (1+o(1)) \frac{8 \Delta_1 C(N) N^{\theta}}{\theta^{2} (\log N)^2}, \quad S_2 \leqslant (1+o(1)) \frac{8 \Delta_2 C(N) N^{\theta}}{\theta^{2} (\log N)^2}, \\
S^{\prime}_1 &\geqslant (1+o(1)) \frac{16 \Delta_3 C(N) N^{\kappa}}{(2\kappa-1) (\log N)^2}, \quad S^{\prime}_2 \leqslant (1+o(1)) \frac{16 \Delta_4 C(N) N^{\kappa}}{(2\kappa-1) (\log N)^2},
\end{align}
where
\begin{align}
\Delta_1 =& \log (5.995 \theta-1)+\int_{2}^{5.995 \theta -2} \frac{\log (s-1)}{s} \log \frac{5.995 \theta-1}{s+1} d s, \\
\Delta_2 =& \log \left(\frac{11.99 \theta-2}{3 \theta-2}\right)+\int_{2}^{5.995 \theta-2} \frac{\log (s-1)}{s} \log \frac{(5.995 \theta-1)(5.995 \theta-1-s)}{s+1} d s, \\
\Delta_3 =& \log (11.99 \kappa-6.995)+\int_{2}^{11.99 \kappa-7.995} \frac{\log (s-1)}{s} \log \frac{11.99 \kappa-6.995}{s+1} d s, \\
\Delta_4 =& \log \left(\frac{23.98 \kappa-13.99}{6 \kappa-5}\right)+\int_{2}^{11.99 \kappa-7.995} \frac{\log (s-1)}{s} \log \frac{(11.99 \kappa-6.995)(11.99 \kappa-6.995-s)}{s+1} d s.
\end{align}
By numerical calculations we get that
\begin{equation}
\Delta_1 - \frac{1}{2} \Delta_2 \geqslant 0.0009 \quad \operatorname{and} \quad \Delta_3 - \frac{1}{2} \Delta_4 \geqslant 0.0009.
\end{equation}
Then by (8)--(16) we have
$$
D_{1,3}^{\theta}(N) \gg \frac{C(N) N^{\theta}}{(\log N)^2} \quad \operatorname{and} \quad D_{1,3}(N,\kappa) \gg \frac{C(N) N^{\kappa}}{(\log N)^2}.
$$
Theorem~\ref{t1} is proved.

\bibliographystyle{plain}
\bibliography{bib}
\end{document}